\documentclass[reqno]{amsart}
\usepackage{graphicx}
\usepackage{trajan}
\usepackage{multirow}
\usepackage{latexsym,array,delarray,amsthm,amssymb,epsfig}
\usepackage{amsmath}
\usepackage{yhmath}
\usepackage{mathdots}
\usepackage{MnSymbol}
\usepackage{tikz}
\usepackage{adjustbox}
\usepackage{rotating}
\usepackage{longtable}
\usepackage{enumitem}
\usepackage{caption}
\usepackage{wasysym}
\usetikzlibrary{shapes.misc,shadows}
\theoremstyle{plain}
\newtheorem{thm}{Theorem}[section]

\newtheorem{prop}[thm]{Proposition}

\newtheorem*{thm*}{Theorem}
\newtheorem*{lemma*}{Lemma}
\newtheorem*{prop*}{Proposition}
\newtheorem*{cor*}{Corollary}
\newtheorem*{conj*}{Conjecture}

\theoremstyle{definition}
\newtheorem{defn}[thm]{Definition}
\newtheorem{ex}[thm]{Example}

\theoremstyle{remark}

\newcommand{\A}{\mathcal{A}}
\newcommand{\C}{\mathcal{C}}

\newcommand{\cM}{\mathcal{M}}
\newcommand{\cL}{\mathcal{L}}

\newcommand{\lrarrow}{\longrightarrow}

\begin{document}
\date{}

\title{Complete Leibniz Algebras}
\author{Kristen Boyle}
\address{Department of Mathematics and Computer Science, Longwood University, Farmville, VA 23909}
\email{boylekj@longwood.edu}
\author{Kailash C. Misra and Ernie Stitzinger}
\address{Department of Mathematics, North Carolina State University, Raleigh, NC 27695-8205}
\email{misra@ncsu.edu, stitz@ncsu.edu}
\thanks{KCM is partially supported by Simons Foundation grant \#  636482}

\begin{abstract}
Leibniz algebras are certain generalization of Lie algebras. It is natural to generalize concepts in Lie algebras to Leibniz algebras and investigate whether the corresponding results still hold. In this paper we introduce the notion of complete Leibniz algebras as generalization of complete Lie algebras. Then we study properties of complete Leibniz algebras and their holomorphs.

\end{abstract}

\maketitle
\bigskip
\section{Introduction}

Leibniz algebras are certain generalization of Lie algebras. In particular, a Leibniz algebra lack the anti-symmetry property of Lie algebras. Leibniz algebras were first studied by Bloh  \cite{bloh} in 1965 and later popularized by Loday \cite{lodayfr}. A left (resp. right) Leibniz algebra $\A$ over field $F$ is a vector space equipped with a bilinear product $[ \, , \,] : \A \times \A \longrightarrow \A$ such that the left (resp. right) multiplication operator is a derivation. Following Barnes \cite{barnes1}, in this paper Leibniz algebras will always refer to finite dimensional left Leibniz algebras. We will follow the notations in \cite{our}. In particular, for $x \in \A$, the left multiplication operator $L_x :\A \lrarrow \A$ is defined by $L_x(y) = [x,y]$ for all $y \in \A$. The vector space of left multiplication operators $L(\A) = \{L_x \mid x \in \A\}$ is a Lie algebra under commutator bracket. Note that a Lie algebra is a Leibniz algebra, but not conversely. An important example of non-Lie Leibniz algebra is the cyclic Leibniz algebra spanned by the powers of a single vector.\\

A Leibniz algebra $\A$ has an abelian ideal ${\rm Leib}(\A)={\rm span}\{x^2=[x, x] \mid x\in \A\}$. The ideal Leib$(\A)= \{0\}$ if and only if $\A$ is a Lie algebra. Furthermore, Leib$(\A)$ is the minimal ideal for which $\A/{\rm Leib}(\A)$ is a Lie algebra. For a Leibniz algebra $\A$, we define the ideals $\A^{(1)} = \A = \A^{1}$, $\A^{(i)} = [\A^{(i-1)}, \A^{(i-1)}]$ and $\A^{i} = [\A, \A^{i-1}]$ for $i \in \mathbb{Z}_{\geq 2}$. The Leibniz algebra is said to be solvable (resp. nilpotent) if $\A^{(m)} = \{0\}$ (resp. $\A^{m} = \{0\}$) for some positive integer $m$. The maximal solvable (resp. nilpotent) ideal of $\A$ is called radical (resp. nilradical) denoted by rad$(\A)$ (resp. nilrad$(\A)$). A Leibniz algebra $\A$ is semisimple if rad$(\A) = {\rm Leib}(\A)$. A linear map $\delta : \A \lrarrow \A$ is a derivation if $\delta [x,x] = [\delta (x),y] + [x, \delta (y)]$. Let $Der(\A)$ be the Lie algebra of all derivations of $\A$. An ideal $B$ of $\A$ is a characteristic ideal if $\delta (B) \subseteq B$ for all $\delta \in Der(\A)$. Many researchers have studied the structure of characteristic ideals of Lie algebras. In the next section we extend some of these results to that of Leibniz algebras.  A Lie algebra $L$  is said to be complete if it has trivial center and all its derivations are inner. Meng \cite{Meng} and his collaborators studied properties of complete Lie algebras. We define a derivation $\delta \in  Der(\A)$ to be inner if there exist $x \in \A$ such that ${\rm im}(\delta -L_x) \subseteq {\rm Leib}(\A)$. We define a Leibniz algebra $\A$ to be complete if the center $Z(\A/{\rm Leib}(\A)) = \{0\}$ and all derivations of $\A$ are inner. As in the case of Lie algebras, we show that nilpotent Leibniz algebras are not complete, but semisimple Leibniz algebras are complete. Finally we define holomorph of a Leibniz algebra and show that the holomorph of a complete Leibniz algebra decomposes like in the case of Lie algebra (see \cite{Meng}). However, unlike in Lie algebra case the converse does not hold for Leibniz algebras.

\bigskip
\section{Characteristic Ideals}
Let $\A$ be a Leibniz algebra and $Der(\A)$ be its derivation algebra. An ideal $I$ of $\A$ is a characteristic ideal if $\delta (I) \subseteq I$ for all 
$\delta \in Der(\A)$. The left center $Z^l(\A) = \{x \in \A\mid [x, a] = 0 \, {\rm for} \, {\rm all} \, a \in \A\}$ is an abelian ideal of $\A$ which contains ${\rm Leib}(\A)$. The following is an easy, but important observation.
\begin{prop}\label{c-Leib} For a finite dimensional Leibniz algebra $\A$ over field $F$, Leib$(\A)$ and $Z^l(\A)$ are characteristic ideals.
\end{prop}

\begin{proof} For $x \in \A$ and $\delta \in Der(\A)$ we have 

$\delta(x^2) = \delta([x,x]) = [\delta(x),x] + [x,\delta(x)] = [\delta(x)+ x, \delta(x) + x] - [\delta(x), \delta(x)] - [x,x] \in {\rm Leib}(\A).$

Hence Leib$(\A)$ is a characteristic ideal. For $a \in  Z^l(\A)$, we have $[a, x] = 0$ for all $x \in \A$. Hence

$0 = \delta([a, x] = [\delta(a), x] + [a, \delta(x)]  =  [\delta(a), x]$

for all $x \in \A$. This implies $\delta(a) \in Z^l(\A)$, proving that $Z^l(\A)$ is a characteristic ideal.

\end{proof}

For a solvable (resp. nilpotent) ideal $I$, we say that its {\it derived length} (resp. {\it nilpotency class}) is $k-1$ if $I^{(k)} = 0$ (resp. $I^k = 0$) and  $I^{(k-1)} \not= 0$ (resp. $I^{k-1} \not= 0$). In the case of a Lie algebra $L$ it is shown that the radical \cite{Petravchuk} (resp. nilradical \cite{Maksimenko}) of $L$  is a characteristic ideal under certain condition on the characteristic of the underlying field $F$. With minor modifications to the Lie algebra proofs we obtain the following results for Leibniz algebras (see \cite{Kristen}).

\begin{thm}\label{radical} Let $\A$ be a Leibniz algebra over field $F$. Then the radical rad$(\A)$ of $\A$ is a characteristic ideal if the characteristic of $F$ is zero or greater than $2^n$ where $n$ is its derived length.
\end{thm} 

\begin{thm}\label{nilradical}  Let $\A$ be a Leibniz algebra over field $F$. Then the nilradical nilrad$(\A)$ of $\A$ is a characteristic ideal if the characteristic of $F$ is zero or greater than $n+1$ where $n$ is its nilpotency class.
\end{thm} 

The following example is due to Jacobson (\cite{Jacobson}, page 75) and Seligman (\cite{Seligman}, page 163), for Lie algebras, hence Leibniz algebras over field $F$ of characteristic $p>2$. It shows that the bounds on the characteristic of the field in Theorem \ref{radical} and Theorem \ref{nilradical} can not be improved.

\begin{ex} Let $F$ be a field of characteristic $p>2$. Let $R$ be the commutative associative algebra over $F$ with basis $\{1, x, x^2, \cdots , x^{p-1}\}$ where $x^p = 0$.Then dim$(R) = p$. The maximal nilpotent ideal $N$ of $R$ has basis $\{x, x^2, \cdots , x^{p-1}\}$. Then dim$(R/N) = 1$. Let $L$ be a simple Lie algebra over $F$. Then $M = L\otimes R$ is a Lie algebra with $[a\otimes y, b\otimes z] = [a,b]\otimes yz$ for all $a,b \in L$ and $y, z \in R$. Since $(L\otimes N)^p = L^p\otimes N^p = L\otimes 0 = 0$, $L\otimes N$ is a nilpotent (hence solvable) ideal of $M$. Since $M/(L\otimes N) = (L\otimes R)/(L\otimes N)$ is isomorphic to $L$, $L\otimes N$ is both the radical and nilradical of $M$.  Consider the derivation $\delta = 1\otimes \frac{d}{dx}$ in $Der(M)$ defined by $\delta (a\otimes x^t) =t(a\otimes x^{t-1}$ for all $a\in L$ and $1\leq t\leq p-1$. Since $\delta (a\otimes x) = a\otimes 1 \not\in L\otimes N$, $L\otimes N$ is not a characteristic ideal of $M$. Since $x^p =0$, the nilpotent class of $L\otimes N$ is $p-1$ and the derived length of $L\otimes N$ is the smallest integer $m$ such that $2^m \geq p$.

\end{ex} 
A Lie algebra is characteristic semisimple if its maximal solvable characteristic ideal is trivial.
The maximal solvable characteristic ideal of the Leibniz algebra $\A$ is called the {\it characteristic radical} of $\A$ and is denoted by $R= Crad(\A)$. Since the left center $Z^l(\A)$ is a characteristic abelian ideal $Crad(\A)$ contains $Z^l(\A)$, hence contains Leib$(\A)$. 
The Leibniz algebra $\A$ is said to be {\it characteristically semisimple} if $Crad(\A) = {\rm Leib}(\A)$. It is said to be {\it characteristically simple} if it's only proper characteristic ideal is Leib$(\A)$ and $\A^2 = \A$.

\begin{thm} Let $R$ be characteristic radical of the Leibniz algebra $\A$. Then $\A/R$ is a characteristic semisimple Lie algebra.
\end{thm}
\begin{proof} Let $S/R$ be a solvable characteristic ideal of $\A/R$. Then $S$ is a solvable ideal of $\A$. Let $\delta$ be a derivation of $\A$. Since $R$ is a characteristic ideal of $\A$, $\delta$ induces a derivation $\bar{\delta}$ on $\A/R$ naturally. Since $S/R$ is a characteristic ideal of $\A/R$, $\bar{\delta} (S/R) \subseteq S/R$ which implies $\delta(S) \subseteq S$. Hence $S$ is a solvable characteristic ideal of $\A$. Thus $S \subseteq R$ which implies $S/R = \{0\}$. Since  Leib$(\A) \subseteq R$, we have the desired result.

\end{proof}

\begin{thm} If the Leibniz algebra $\A$ is characteristacally semisimple (resp. simple) then the Lie algebra $\A/{\rm Lieb}(\A)$ is characteristically semisimple (resp. simple).

\end{thm}
\begin{proof}
Suppose $\A$ is a characteristacally semisimple Leibniz algebra. Then  $Crad(\A) = {\rm Leib}(\A)$. Hence $Crad(\A/{\rm Leib}(\A)) = \{0\}$ which implies that $\A/{\rm Leib}(\A)$ is a characteristacally semisimple Lie algebra.

Now suppose $\A$ is a characteristically simple Leibniz algebra. If ${\rm Leib}(\A) = \{0\}$, then clearly $\A$ is  a characteristacally simple Lie algebra. Assume ${\rm Leib}(\A) \not= \{0\}$. Then ${\rm Leib}(\A)$ is the only proper characteristic ideal of $\A$ and  $\A^2 = \A$. This implies that the Lie algebra $\A/{\rm Leib}(\A)$ has no proper nonzero characteristic ideals and $(\A/{\rm Leib}(\A))^2 = \A/{\rm Leib}(\A)$. Hence 
the Lie algebra $\A/{\rm Leib}(\A)$ is characteristically simple.
\end{proof}

\begin{defn} The Leibniz algebra is said to be {\it completely semisimple} if it is a direct sum of ideals, each of which is characteristically simple.
\end{defn}

\begin{thm} Every completely semisimple Leibniz algebra is characteristic semisimple.

\end{thm}

\begin{proof} Let $\A$ be a completely semisimple Leibniz algebra. Then $\A = \A_1 \oplus \A_2 \oplus \cdots \oplus \A_k$ where each $\A_i$ is characteristically simple. Then clearly Leib$(\A) = \oplus_{i=1}^k$Leib$(\A_i)$. Let $B$ be a solvable characteristic ideal of $\A$. Then $D = B+$ Leib$(\A)$ is a solvable characteristic ideal of $\A$. For $1 \le i \le k$, the projection map from $\A$ to $\A_i$ sends $D$ onto a solvable characteristic ideal $D_i$ of $\A_i$ since each derivation $\delta_i$ of $\A_i$ extends to a derivation $\delta$ of $\A$ by defining $\delta (A_j) = 0$ for $j \not= i$. Since each $\A_i$ is characteristically simple we have $D_i \subseteq$ Leib$(\A_i)$, $1 \le i \le k$. Hence $D = B+$ Leib$(\A) \subseteq$ Leib$(\A)$ which implies $B \subseteq$ Leib$(\A)$. Therefore, $\A$ is characteristically semisimple.

\end{proof}

\section{Complete Leibniz Algebra}
A Lie algebra is said to be complete if it has trivial center and all its derivations are inner. It is known that semisimple Lie algebras over a field of characteristic zero is complete, while nonzero nilpotent Lie algebras are never complete since it has nontrivial center and it also has outer derivations \cite{Jacobson-1}. Meng \cite{Meng} has characterized complete Lie algebras in terms certain conditions  on its holomorph. Motivated by these results, we define complete Leibniz algebras below and study their properties.
 
Let $\A$ be a Leibniz algebra over field $F$ and $Der(\A)$ be its derivation algebra. We say that a derivation $\delta \in Der(\A)$ is inner if there exists $x \in \A$ such that ${\rm im}(\delta - L_x) \subseteq {\rm Leib}(\A)$.

\begin{defn} The Leibniz algebra $\A$  is said to be {\it complete} if 
\begin{enumerate}
\item the center $Z(\A/{\rm Leib}(\A)) = \{0\}$, and 
\item every derivation of $\A$ is inner.
\end{enumerate}
\end{defn}

Observe that $Z(\A/{\rm Leib}(\A)) = \{0\}$ implies $Z^l(\A) = {\rm Leib}(\A)$, but not conversely. For example, consider the three dimensional 
cyclic Leibniz algebra $\A = {\rm span}\{x, x^2,x^3\}$ with $[x,x] = x^2, [x,x^2] = x^3, [x,x^3] = 0$. Then $Z^l(\A) = {\rm Leib}(\A) = {\rm span}\{x^2,x^3\}$, but  $Z(\A/{\rm Leib}(\A)) \not= \{0\}$. If a Lie algebra $\A$ is complete as a Leibniz algebra, then it is also complete as a Lie algebra since ${\rm Leib}(\A) =\{0\}$ in this case. Also note that if $\A$ is a complete Leibniz algebra, then the Lie algebra $\A/{\rm Leib}(\A)$ is non-abelian.

\begin{prop} For a Leibniz algebra $\A$ if the Lie algebra $\A/{\rm Leib}(\A)$ is complete, then $\A$ is a complete Leibniz algebra.
\end{prop}

\begin{proof}
Suppose $\A/{\rm Leib}(\A)$ is a complete Lie algebra. Then $Z(\A/{\rm Leib}(\A)) = \{0\}$ and all the derivations of $\A/{\rm Leib}(\A)$ are inner. Let $\delta \in Der(\A)$. By Proposition \ref{c-Leib} , $\delta ({\rm Leib}(\A)) \subseteq {\rm Leib}(\A)$. So $\delta$ naturally induces a derivation $\bar{\delta} \in Der(\A/{\rm Leib}(\A))$. Hence there is $\bar{x} = x + {\rm Leib}(\A) \in \A/{\rm Leib}(\A)$ such that $\bar{\delta} = L_{\bar{x}}$. This implies that 
${\rm im}(\delta - L_x) \subseteq {\rm Leib}(\A)$, which proves the proposition.
\end{proof}

We say that the Leibniz algebra $\A$ is semisimple if $rad(\A) = {\rm Leib}(\A)$ (see \cite{Omirov}). Recall that over field of characteristic zero, semisimple Lie algebras are complete. As shown below the corresponding result holds for Leibniz algebras.

\begin{thm} Let $\A$ be a semisimple Leibniz algebra over field $F$ of characteristic zero. Then $\A$ is a complete Leibniz algebra.

\end{thm}

\begin{proof} 
Let $\A$ be a semisimple Leibniz algebra over a field $F$ of characteristic zero. Then $\A = (S_1 \oplus S_2 \oplus \cdots \oplus S_k) + {\rm Leib}(\A)$ where each $S_j$ is a simple Lie algebra \cite{GKO}. Hence $\A/{\rm Leib}(\A)$ is a semisimple Lie algebra which is a complete Lie algebra. So $Z(\A/{\rm Leib}(\A)) = \{0\}$. Furthermore, if $\delta \in Der(\A)$ then it induces a derivation $\bar{\delta}$ on $\A/{\rm Leib}(\A)$ since ${\rm Leib}(\A)$ is a characteristic ideal. Therefore, there exists $\bar{x} = x+ {\rm Leib}(\A) \in \A/{\rm Leib}(\A)$ such that $\bar{\delta} = {\rm ad}_{\bar{x}}$ which implies that ${\rm im}(\delta - L_x) \subseteq {\rm Leib}(\A)$. This proves that $\A$ is a complete Leibniz algebra.
\end{proof}

\begin{prop} Let $\A$ be a nilpotent Leibniz algebra. Then $\A$ is not complete.
\end{prop}
\begin{proof} 
Let $\A$ be a nilpotent Leibniz algebra. Then $\A/{\rm Leib}(\A)$ is a nilpotent Lie algebra. Hence  $Z(\A/{\rm Leib}(\A)) \not= \{0\}$ which implies that $\A$ is not complete.
\end{proof}

Suppose $\C$ be a cyclic $n$-dimensional  Leibniz algebra over field $F$ generated by a nonzero element $x \in \C$. Then $\C = {\rm span}\{x, x^2, \cdots , x^n\}$ with $[x, x^n] = \sum_{i = 2}^n k_ix^i$ for some $k_2, k_3, \cdots k_n \in F$. It is easy to see that  $\C$ is nipotent if $k_j = 0, \, 2 \leq j \leq n$ otherwise $\C$ is a solvable Leibniz algebra. Also ${\rm Leib}(\C) = \C^2 = {\rm span}\{x^2, \cdots , x^n\}$.
Hence $\C/{\rm Leib}(\C)$ is a one dimensional Lie algebra which implies that $\C$ is not complete. The following propositions give explicit descriptions of $Der(\C)$ by direct calculations (see \cite{Kristen}). It is clear from these explicit descriptions that a nilpotent cyclic Leibniz algebra does contain outer derivations, whereas all derivations for a non-nilpotent solvable Leibniz algebra are inner.

\begin{prop} Let $\C = {\rm span}\{x, x^2, \cdots , x^n\}$ with $[x, x^n] = 0$ be the $n$-dimensional nilpotent cyclic Leibniz algebra. Then $Der(\C)$ is a $n$- dimensional Lie algebra spanned by $\{\delta_1 , \delta_2 , \cdots , \delta_n\}$ where
\begin{equation*}
\delta_k(x^i) = \begin{cases} 
ix^i , \, {\rm for} \, k=1, \, 1 \leq i \leq n\\
x^{i+k-1}, \, {\rm for} \, \, i+k-1 \leq n, \, 2 \leq k \leq n\\
0, \quad {\rm otherwise}.
\end{cases}
\end{equation*}

\end{prop}

\begin{prop} Let $\C = {\rm span}\{x, x^2, \cdots , x^n\}$ with $0 \not= [x, x^n] = \sum_{i = 2}^n k_ix^i$ be the $n$-dimensional non-nilpotent cyclic Leibniz algebra. Then $Der(\C)$ is a $(n-1)$- dimensional Lie algebra spanned by $\{\delta_1 , \delta_2 , \cdots , \delta_{n-1}\}$ where
\begin{equation*}
\delta_k(x^i) = \begin{cases} 
x^{i+1} , \, {\rm for} \, k=1, \, 1 \leq i \leq n-1\\
 \sum_{i = 2}^n k_ix^i , \, k=1 , \, i=n\\
\delta_1^k(x^i), \, {\rm for} \, \, 1 \leq i \leq n, \, 2 \leq k \leq n-1.
\end{cases}
\end{equation*}

\end{prop}

It is known that a nilpotent Lie algebra does contain outer derivations \cite{Jacobson-1}. However, as the following example shows 
there exists nilpotent Leibniz algebras which does not have outer derivations.

\begin{ex}\label{nil-inner} 
Consider the Leibniz algebra $\A = {\rm span}\{w, x, y, z\}$ with non-zero multiplications $[x, x] = z, [w, x] = y, [x, w] = -y+z, [w, y] = z, [y, w] = -z$.
Clearly ${\rm Leib}(\A) = {\rm span}\{z\} \subseteq \A^2 = {\rm span}\{y, z\}$ and $\A^4 = \{0\}$. So $\A$ is nilpotent. By direct calculation it is easy to see that $Der(\A) = {\rm span}\{\delta_1, \delta_2, \delta_3, \delta_4\}$ where

\vspace{3mm}
$\delta_1(w)=y, \delta_1(x)=0, \delta_1(y)=0, \delta_1(z)= 0,$

$\delta_2(w)=z, \delta_2(x)=0, \delta_2(y)=0, \delta_2(z)= 0,$

$\delta_3(w)=0, \delta_3(x)=y, \delta_3(y)=z, \delta_3(z)= 0,$

$\delta_4(w)=0, \delta_4(x)=z, \delta_4(y)=0, \delta_4(z)=0.$

\vspace{5mm}
\noindent Note that $L_w = \delta_3, L_x = -\delta_1+\delta_2+\delta_4$ and $L_y = -\delta_2$.  By definition im($\delta_i) \subseteq {\rm Leib}(\A)$ for $i = 2,  4$.  Also  im$(\delta_3-L_w) \subseteq {\rm Lieb}(\A)$ and  im$(\delta_1+ L_x$)= im($\delta_2+\delta_4) \subseteq {\rm Leib}(\A)$. Hence by linearity all derivations of $\A$ are inner.
\end{ex}

\section{Holomorph of Leibniz Algebra}. 
Let $\A$ be a Leibniz algebra over field $F$ and $Der(\A)$ be the Lie algebra of derivations of $\A$. Following the definition of Holomorph of a Lie algebra (see \cite{Jacobson}), we define the Holomorph of the Leibniz algebra $\A$ to be the vector space $hol(\A) = \A \oplus Der(\A)$, with multiplication defined by $[x+\delta_1, y+\delta_2] = [x,y] + \delta_1(y) + [L_x, \delta_2] + [\delta_1, \delta_2]$ for all $x, y \in \A, \delta_1, \delta_2 \in Der(\A)$. Observe that $[L_x, \delta_2] = L_{- \delta_2(x)}$. It is easy to check that $hol(\A)$ is a Leibniz algebra. For two subspaces $\cM$ and $\cL$ of $hol(\A)$ we define the left centralizer of $\cM$ in $\cL$ to be $Z_{\cL}^l(\cM) = \{x \in \cL \mid [x, \cM] = 0\}$.

\begin{prop}\label{hol-1} 
$Z_{hol(\A)}^l(\A) = \{x - L_x \mid x \in \A\}$.
\end{prop}
\begin{proof} Clearly $ \{x - L_x \mid x \in \A\} \subseteq Z_{hol(\A)}^l(\A)$.  Let $z \in Z_{hol(\A)}^l(\A)$. Then $z = x+\delta$ for some $x \in \A, \delta \in Der(\A)$. Then $0 = L_z(y) = [x+\delta , y] = L_x(y) + \delta(y)$ for all $y\in \A$. This implies $\delta = - L_x$, hence $z = x - L_x$ which proves the statement. 

\end{proof}

\begin{prop}\label{hol-2} 
$A\cap Z_{hol(\A)}^l(\A) = Z^l(\A)$.
\end{prop}
\begin{proof} Clearly $ Z^l(\A) \subseteq A\cap Z_{hol(\A)}^l(\A)$.  Let $z \in A\cap Z_{hol(\A)}^l(\A)$. Then by Proposition \ref{hol-1}, $z = x- L_x \in \A$ for some $x \in \A$. This implies $z = x$ and $L_x = 0$. So $A\cap Z_{hol(\A)}^l(\A) \subseteq \{x \in \A \mid L_x =0\} = Z^l(\A)$ proving the result.

\end{proof}

In \cite{Meng}, Meng proved that a Lie algebra $L$ is complete if and only if  $hol(L)=L\oplus$(Z$_{hol(L)}(L)$. The following theorem is the corresponding result in the case of Leibniz algebras.

\begin{thm}\label{complete} If the Leibniz algebra $\A$ is complete then $hol(\A) = \A + (Z_{hol(\A)}^l(\A) \oplus I)$ and $\A \cap (Z_{hol(\A)}^l(\A) \oplus I) = {\rm Leib}(\A)$, where $I = \{\delta \in Der(\A) \mid {\rm im}(\delta) \subseteq {\rm Leib}(\A)\}$.

\end{thm} 

\begin{proof}
Let $x + \delta \in hol(\A)$ be any element. Then $x \in \A$ and $\delta \in Der(\A)$.  Since $\A$ is  complete there is some $y \in \A$ such that 
${\rm im}(\delta - L_y) \subseteq {\rm Leib}(\A)$. Then $(\delta - L_y) \in I$.  So $x + \delta = (x+y) + (-y - L_{-y}) + (\delta - L_y) \in \A + Z_{hol(\A)}^l(\A) + I$ by Proposition \ref{hol-1}.

To show that the second sum is direct, let $\delta \in Z_{hol(\A)}^l(\A)\cap I$. Then $[\delta , x] = \delta(x) = 0$ for all $x \in \A \subseteq hol(\A)$. Hence $\delta = 0$ which implies  $Z_{hol(\A)}^l(\A)\cap I = \{0\}$. Finally, by Proposition \ref{hol-2} ,  $\A \cap (Z_{hol(\A)}^l(\A) \oplus I) = 
\A \cap Z_{hol(\A)}^l(\A) = Z^l(\A) = {\rm Leib}(\A)$ since $\A$ is complete.

\end{proof}

Consider the nilpotent Leibniz algebra $\A = {\rm span}\{w, x, y, z\}$ with non-zero multiplications $[x, x] = z, [w, x] = y, [x, w] = -y+z, [w, y] = z, [y, w] = -z$ in Example \ref{nil-inner} which is not complete. In this case, ${\rm Leib}(\A) = {\rm span}\{z\}, I = {\rm span}\{\delta_2, \delta_4\}, Z_{hol(\A)}^l(\A) = {\rm span}\{z, w- L_w, x- L_x, y- L_y\} = {\rm span}\{z, w- \delta_3, x+ \delta_1- \delta_2- \delta_4, y+ \delta_2\}$. Hence 
$Z_{hol(\A)}^l(\A)\cap I = \{0\}$ and $hol(\A) = \A + (Z_{hol(\A)}^l(\A) \oplus I)$. Also $\A \cap (Z_{hol(\A)}^l(\A) \oplus I) = {\rm Leib}(\A)$.
Thus the converse of the statement in Theorem \ref{complete} does not hold. However, it is shown in Example \ref{nil-inner} , all derivations of $\A$ are inner. The following result shows that it is true in general. 

\begin{thm}\label{complete} Let $\A$ be a Leibniz algebra such that $hol(\A) = \A + (Z_{hol(\A)}^l(\A) \oplus I)$, where $I = \{\delta \in Der(\A) \mid {\rm im}(\delta) \subseteq {\rm Leib}(\A)$. Then all derivations of $\A$ are inner.

\end{thm} 

\begin{proof} Let $\delta \in Der(\A) \subseteq hol(\A) = \A + (Z_{hol(\A)}^l(\A) \oplus I)$. Then by Proposition \ref{hol-1} $\delta = x + (y- L_y) + \delta_1$ for some $x, y \in \A, \delta_1 \in  I$.  Then for any $z \in \A$ we have $\delta(z) = [\delta, z] = [x, z] + [y, z] - [y, z] + [\delta_1, z] = L_x(z) + \delta_1(z)$. Therefore, $(\delta - L_x)(z) = \delta_1(z) \in {\rm Leib}(\A)$ which implies that $\delta$ is a inner derivation.

\end{proof}

\end{document}